\documentclass[12pt]{amsart}
\usepackage{latexsym}
\usepackage{amsfonts}
\usepackage{amsthm}
\usepackage{amsmath}
\usepackage{amssymb}

\def\<{{\langle}}
\def\>{{\rangle}}

\def\note#1{{}}

\def\note#1{}

\def\beq{\begin{equation}}
\def\eeq{\end{equation}}

\newcounter{zlist}

\def\Label#1{\label{#1}\ifmmode\llap{[#1] }\else
\marginpar{\smash{\hbox{\tiny [#1]}}}\fi}
\def\Label{\label}

\newtheorem{proposition}{Proposition}[section]

\newtheorem{theorem}[proposition]{Theorem}

\theoremstyle{definition}
\newtheorem{definition}[proposition]{Definition}

\theoremstyle{remark}
\newtheorem{remark}[proposition]{Remark}

\newcounter{c}

\newcommand{\etyk}[1]{\vspace{-7.4mm}$$\begin{equation}\Label{#1}
\addtocounter{c}{1}}
\renewcommand{\]}{\ifnum \value{c}=1 $$\else \end{equation}\fi}
\begin{document}

\title{Meeting Solomon Marcus}
\author{Florin Felix NICHITA}
\address{Institute of Mathematics of the Romanian Academy, 
P.O. Box 1-764, RO-70700, Bucharest, Romania}
\date{}
%\keyword{racks; series; Euler's identity;
%Euler's relation; memories }
%\subjclass{00A05;  01A05; 11D09; 16B50; 30D05; 33B10}
\begin{abstract} { Dedicated to
Solomon Marcus, the current paper continues a recent series
about our meetings. Trying to recreate
the spirit of those meetings, we first propose a
discussion which started as  a high-school problem.
The main part of the current paper consists in
a section about racks.
It
 presents
elements of trigonometry  in racks, and Euler formulas
associated in this framework. }
\end{abstract}
\maketitle
%\keyword{racks; series; Euler's identity;
%Euler's relation; memories }
 


%\MSC{00A05;  01A05; 11D09; 16B50; 30D05; 33B10}
\section{Introduction }
 
I received an invitation from the journal
``Libertas Matematica 
(new series)'' to write an article
about Solomon Marcus (see, for example, \cite{MeSM, editor})
on the occasion of one hundred years
from his birth (March 1-st, 1925).
So, I thought about this invitation,
I recalled our meetings, I discussed with several people
about Solomon Marcus, and I decided that
it was the right moment to
publish some of my memories.
I
submitted my manuscript (\cite{Memo}), and it
 was 
accepted soon afterwards.
This determined me to write a second paper
dedicated to Solomon Marcus. Writing
about such a huge  personality  is not an 
easy task, but certainly a rewarding one,
especially when you recall some important moments from
our professional activity.
When my second paper (\cite{mMemo}) was still
in the process of being evaluated for publication
in ``Libertas Matematica 
(new series)'', I posted a survey of  \cite{Memo} and \cite{mMemo} 
on the 
Mathematical Archive (\cite{mMemo2}), adding
some new content. For example, I
gave applications of the B-ring Euler formula in
finding solutions for the braid condition.
This was well-received, and my e-mail correspondence
led me to the remark that Euler, Erdos and Euclid
are frequently mentioned in my discussions with Solomon Marcus. 
Also, there were  interesting interactions
with journals, colleagues etc.
From Brown University (Rhode Island), I received four
articles about Solomon Marcus (see \cite{s1,s2,s3,s4}).
Needless to say how much I valuated them.
I learned about: the life of Solomon Marcus in Providence, the
``junior'' Solomon Marcus meeting Al. Rosetti, his work at
``Poetica Mathematica'', the correspondence between Solomon
Marcus and Sanda Golopentia, the relationship between science and
literature, the contributions of Solomon Marcus to the literature
journals (Romania Literara, Viata Romaneasca, Dilema Veche etc) 
and so on.
These four articles have motivated me to write  
the fourth article in the series which starting with the
paper \cite{Memo}.
Also, I would like to
mention another Mathematical Archive communication (\cite{arxive}),
which was posted soon after our preprint (\cite{mMemo2}), and it gives some weight
to our results, because the solutions to the braid equation
lead to representations of the braid group. More recently,
Bradshaw and Vignat wrote a paper, in a similar manner
with ours, about another
``beautiful mind'' 
(see \cite{bv}).

This time, the paper in our series, starts
 at the mathematical level
with
 some discussions with high-school students
from the best institutions from Bucharest at some
informal meetings at 
{\em Math Cafe} in 2026.
We began our chat with a high-school problem. As we continued to develop
this discussion we were led to a very hard problem.
We were not sure if that problem has any solution at all,
but
it eventually led to a sophisticated
approach to unify the arithmetic and geometric series.
In the last part of the current paper 
we introduce
elements of trigonometry  in racks and obtain an Euler formula
in this context.
So, the mathematical
content is a continuation of \cite{mMemo2}. Further readings could
be the following: \cite{sm2,sm3,e,b2,bw,fn, tc}.

\bigskip

\section{Series and   Sequences}

\bigskip

This section is based on some discussions at {\em Math Cafe} in 2026.
We started our chat on a high-school problem. As we continued to develop
this discussion we were led to a very hard problem:
``Is it possible to unify the geometric and arithmetic sums,
their evaluating formulas and their methods of proof ?''
We were not sure that this problem has any solution,
but
it eventually led to a sophisticated
approach to unify the arithmetic and geometric series.
We present and prove three theorems,  
we give some examples,
and then we give
the general theorem.

\bigskip

\begin{theorem}
 Let $  A \in {M}_2 (\mathbb{R}) $ with $ \ \det (A) = 1 $.
Then, $ \ \sum^3_{k=1} \  A^k = ( tr(A) + 1)  \ A^2 \ $.
\end{theorem}

{\bf Proof.} We use the formula $ A^2 - tr(A) A + I_2 = O_2 \ $.
So,
$ \ \sum^3_{k=1} \  A^k = (A + A^3) + A^2 =
tr(A) A^2 + A^2 = ( tr(A) + 1)  \ A^2 \ $.
\qed

\bigskip

\begin{theorem}
 Let $ \ A \in {M}_2 (\mathbb{R}) $ such that $ \ \det (A) = 1 $.

Then, $ \ \sum^9_{k=1} \  A^k = ( tr(A) + 1) ( tr(A^3) + 1) \ A^5 \ $.
\end{theorem}

{\bf Proof.} We use the formula $ A^2 - tr(A) A + I_2 = O_2 \ $.
So,
$ \ \sum^9_{k=1} \  A^k = (A + A^3) + A^2 +
(A^4 + A^6) + A^5 + (A^6 + A^8) + A^7 = 
( tr(A) + 1)  \ A^2 + ( tr(A) + 1)  \ A^5 +
 ( tr(A) + 1) \ A^8 = 
 ( tr(A) + 1) ( tr(A^3) + 1) \ A^5 \ $.
\qed

\bigskip

\begin{theorem}
 Let $ \ A \in {M}_2 (\mathbb{R}) $ such that $ \ \det (A) = 1 $.

Then, $ \ \sum^{27}_{k=1} \  A^k = ( tr(A) + 1) ( tr(A^3) + 1) 
 ( tr(A^9) + 1) 
\ A^{14} \ $.
\end{theorem}

{\bf Proof.} 
$ \ \sum^{27}_{k=1} \  A^k = (A + A^3) + A^2 +
(A^4 + A^6) + A^5 + (A^6 + A^8) + A^7 + . . . +
(A^{25} + A^{27}) + A^{26}
= ( tr(A) + 1) ( tr(A^3) + 1) 
 ( tr(A^9) + 1) 
\ A^{14} \ $.
\qed

\bigskip

{ \em The following examples
are an illustration for the unification of 
the sum of non-zero digits and a corresponding 
geometric series.}

\begin{equation} \label{rmatcon2}
A= \begin{pmatrix}
1 & 1 \\
0 & 1 \\
\end{pmatrix}
\end{equation}

\begin{equation} \label{rmatcon2}
B= \begin{pmatrix}
2 & 0 \\
0 & \frac{1}{2} \\
\end{pmatrix}
\end{equation}

\bigskip

{ \em The following examples
shows how important is to use our formula.}

\begin{equation} \label{rmatcon2}
Let \ A = \begin{pmatrix}
1 & -2 \\
-1 & 3 \\
\end{pmatrix}
\end{equation}

Then,

\begin{equation} \label{rmatcon2}
\sum^9_{k=1}   A^k =  \ 265 \begin{pmatrix}
153 & -418 \\
-208 & 571 \\
\end{pmatrix}
\end{equation}

\bigskip

\begin{theorem}
 Let $ \ A \in {M}_2 (\mathbb{R}) $ such that $ \ \det (A) = 1 $.

Then, $ \ \sum^{3^n}_{k=1} \  A^k = 
( tr(A) + 1) ( tr(A^3) + 1) ... ( tr(A^{3^{n-1}}) + 1)
\ A^{\frac{3^n + 1}{2}} \ $.
\end{theorem}

{\bf Proof.} 
$ \ \sum^{3^n}_{k=1} \  A^k =  (A + A^3) + A^2 +
(A^4 + A^6) + A^5 + (A^6 + A^8) + A^7 + . . . +
(A^{{3^n} -2} + A^{{3^n}}) + A^{{3^n} - 1}
= 
( tr(A) + 1) ( tr(A^3) + 1) ... ( tr(A^{3^{n-1}}) + 1)
\ A^{\frac{3^n + 1}{2}} \ $.
\qed

\bigskip

{ \em The following examples
are  illustrations for the above theorem.}

\begin{equation} \label{rmatcon2}
Let \ A= \begin{pmatrix}
1 & 1 \\
0 & 1 \\
\end{pmatrix}
\end{equation}

 $ \ \sum^{81}_{k=1} \  A^k = 
3 \times 3 \times 3 \times 3 \ \begin{pmatrix}
1 & 41 \\
0 & 1 \\
\end{pmatrix}
$
which corresponds to the result for the 
corresponding arithmetics series.

\bigskip

\begin{equation} \label{rmatcon2}
Let \ B= \begin{pmatrix}
2 & 0 \\
0 & \frac{1}{2} \\
\end{pmatrix}
\end{equation}

 Then $ \ \sum^{81}_{k=1} \  B^k = 
3.5 \times \frac{64+1+8}{8}  \times \frac{2^{18} +1+2^9}{2^9} 
\times \frac{2^{54}+1+2^{27}}{2^{27}} \ \begin{pmatrix}
2^{41} & 0 \\
0 & \frac{1}{2}^{41} \\
\end{pmatrix}
$
which corresponds to the result for the corresponding geometric series.

\bigskip

\section{Trigonometry  and the Euler formula in racks}
%Commentaries
%and Clarifications} 

%\bigskip

In this section,
we will introduce
elements of trigonometry  in racks (see \cite{racks}). 
The terminology might be slightly changed from the 
original theory. These adjustments will help us to
emphasize the new results. 
We will obtain 
 Euler formulas and identities
in this framework.

\subsection{Trigonometry   in racks}

\begin{definition}
 A rack is triple $ ( S, \ \cdot, \ \diamond \ )$, where
$ S $ is
a set with two binary operations,
satisfying the following  four axioms:

$ a \cdot (b \cdot c) = (a \cdot b) \cdot (a \cdot c) \ , $
$  \ \ \  \ \ \  \ \ \ \ \ \ \ \ \ \ \ \ (a \cdot b) \  \diamond  \  a = b \ ,$

$ 
a \cdot (b \ \diamond  \  a)  = b \ $,
$ \ \ \ \ \ \ \ \ \ \ 
 ( c \ \diamond \  b) \ \diamond  \ a = 
( c \ \diamond  \  a) \  \diamond  \  ( b \ \diamond  \  a) \ .$

The operation $ \  \cdot \ $ is called the main operation, we will
write $ a \cdot b = ab \ $, and it has priority over $ \ \diamond  \ $
 in formulas.
\end{definition}

 %\bigskip

\begin{remark}
 The following is an example of rack associated to a group:
 
$ \ \ ab = aba^{-1} $, $ \ \ a \ \diamond  \   b = b^{-1}ab $.
\end{remark}

\bigskip

We now choose $ \ \ \ \ e, \ O  \in S$, and let
$ \ \ \Pi = eO \ \ $ and 
$  \ U = e(eO) $.

Let $ \cos x = ex \ , \ \sin x = x \ \diamond  \  e \ \ $ be 
``trigonometric'' functions
in our rack.

%\bigskip

The following properties hold for our ``trigonometric'' functions:

$ \ \ \ \ \ \cos \Pi = U \ ; \ \ \ \ \ \ \ \sin \Pi = O \ $;

%\subsection{A little history} 

$ \cos \ xy =  \cos x \ \cos y \ ;
\ \ \ \ \ \ \ \ \ \ \ \ \ \ 
 \cos (x \ \diamond  \ y) =  \cos x \ \diamond  \  \cos y \ $;

$ \sin \ x   y = \sin x \ \sin y \ ;
\ \ \ \ \ \ \ \ \ \ \ \ \ \ 
\sin (x \ \diamond  \ y) = \sin x \ \diamond  \  \sin y \ $.

The fundamental formula for this ``trigonometry'' is the following:
$$  \sin \ ( \cos x ) \ = \cos ( \sin x  ) = x. $$ 

\bigskip

\begin{remark}
 The following rack can be defined now:
$ \ \ ab = \cos b $, $ \ \ a \diamond   b = \sin a $.
\end{remark}

\subsection{Weak racks}

\begin{definition}
 A { \bf weak rack} is triple $ ( S, \ \cdot, \ \diamond \ )$, where
$ S $ is
a set with two binary operations,
satisfying the following  three axioms:

$ \ \  \ a \cdot (b \cdot c) = (a \cdot b) \cdot (a \cdot c) \ , $

$  
\ \  (a \cdot b) \  \diamond  \  a = a \cdot (b \ \diamond  \  a) \ ,$

$ \ \ 
 ( c \ \diamond \  b) \ \diamond  \ a = 
( c \ \diamond  \  a) \  \diamond  \  ( b \ \diamond  \  a) \ .$

Again, the for the main operation 
 we will
write $ a \cdot b = ab \ $, and it has priority over $ \ \diamond  \ $
 in formulas.
\end{definition}

\begin{remark}
 The following weak rack can be associated to a Boolean algebra:
$$ \ \ ab = a \rightarrow b ,  \ \ a \diamond   b =  a \setminus b \ . $$

Also, the following weak rack can be associated:
$ \ \ ab = a \vee b ,  \ a \diamond   b =  a \wedge b  . $
\end{remark}

\bigskip

Let $  \ e, \ O  \in S$, 
$  \ \Pi = eO  $, 
$  \ U = e(eO) $,
$ \cos x = ex \ ,$ and $ \ \sin x = x \ \diamond  \  e $. We have
the following properties:
$ \ \ \ \ \ \cos \Pi = U \ ; \ \ \ \ \ \ \ \sin \Pi = O \ $;
$ \cos \ xy =  \cos x \ \cos y \ ; $

$ \cos (x \ \diamond  \ y) =  \cos x \ \diamond  \  \cos y \ $;
$ \ \ \sin \ x   y = \sin x \ \sin y \ ;
\ \ \ \ \ \ \ \ \ \ \ \ \ \ 
\sin (x \ \diamond  \ y) = \sin x \ \diamond  \  \sin y \ $.

The fundamental formula is the following:
$  \ \ \sin \ ( \cos x ) \ = \cos ( \sin x  ) . $

\subsection{The Euler formula and the Euler identity   in 
(weak) racks}

\begin{definition} We can define a dual rack
with the opposite operations,
$ ( S, \ * \ ,  \bullet \ )$,
where the new operations are the following:
$ \ \ \ \  a * b = b \ \diamond \  a \ $ and 
$ \ \ a \bullet b = b \cdot a \ $.

\end{definition}

\begin{remark}
 The rack having the following operations is self-dual:
$ \ \ ab = b $, $ \ \ a \diamond   b = a $.
\end{remark}

On the set $ \ S \times S $, we can put a rack structure
obtained
as the  product of the initial rack with its dual.
Let us denote its main operation as follows:\\
$  \ \ {(x, \ y) \ \square \ (u, \ v) = (xu, \ v \diamond y)
= (xu, \  y * v)} $.

%\bigskip

 For $ a \in S$, we  define 
an ``exponential'' function on $ \ S \times S $:\\
$ \ \  \exp_{a} (x, \ y) =
a^{(x, \ y)}  = (ax, \ a * y) $.

%\newpage

%\bigskip

The ``exponential'' map has the property:
$ \ a^{(x, \ y) \square   (u, \ v)} = 
\ a^{(x, \ y)} \ \square \  a^{(u, \ v)}
 $.

The diagonal map, $ \ \Delta : S \rightarrow S \times S, \
x \mapsto ( x, \ x) \ $,  is a rack morphism for a certain rack structure
on $ \ S \times S$.

\bigskip

\begin{theorem} ({\bf Euler formula in racks.})
The following formula holds:
\begin{equation} 
\exp_{e} \circ \ \Delta = 
[ \cos  \ \times \ \sin ] \circ \Delta \ .
\end{equation}
Equivalently,
\begin{equation} 
{e}^{(x,\ x)} = 
( \cos x ,  \  \sin x )  .
\end{equation}

Moreover, the following identity is true:
$ \ \ \ \ \ \  e^{\Delta (\Pi )} = ( U,\ O ) \ . $

\end{theorem}

{ \bf Proof.} The left hand side reads:
$ \ \ \exp_{e} \circ \ \Delta (x) = e^{(x, \ x)} =
(ex, \ e*x) \ . $
The right hand side reads:
$ \ \ [ \cos  \ \times \ \sin ] \circ \Delta (x) 
= (\cos x, \ \sin x) \ . $
So, the left hand side equals the right hand side, because
$ \ \cos x = ex $ from the definition, and 
$ \ \sin x = x \diamond e = e*x \ $.

For the last identity we have:
$  e^{\Delta (\Pi )} =  e^{(\Pi , \ \Pi )} 
= ( e \Pi, \ e* \Pi) =  ( U,\ O )$.
\qed

\begin{remark}
The classical Euler's formula states that:
\begin{equation} \label{f}
  e^{ ix } = \cos x + i \sin x \ \ \ \ \ \forall x \in \mathbb{R}.
\end{equation} 

This formula could be also written as:
\begin{equation} 
\exp_{e} \circ \ j = 
[ \cos  \ \times \ \sin ] \circ \Delta \ ,
\end{equation}
where $ \ \ j : \mathbb{R} \rightarrow \mathbb{R} \times \mathbb{R} \ ,
\ \ x \mapsto (0, \ x) \ . $
For $x= \pi$ the Euler's formula becomes the Euler's identity. 
\end{remark}

\bigskip

We  consider the following ``hyperbolic'' functions:\\
$ \ \cosh (x, y) = (ex,  y) \ $ and
$ \ \ \sinh (x, \ y) = (x, e * y) $.

\bigskip

It is easy to check the hyperbolic functions Euler formula:
$$ \  \ \exp_e  \ = \ \cosh \ \circ \ \sinh \ = 
\ \sinh  \  \circ \ \cosh  \ . $$

\bigskip

\begin{theorem}
The functions $ \ \exp_e \ , \ \cosh $ and $ \ \sinh $ are
solutions for the Quantum Yang-Baxter equation 
($ R^{12} \circ  R^{13} \circ R^{23} \ = \  
R^{23} \circ  R^{13} \circ R^{12} \  $). 
\end{theorem}

{\bf Proof. $ \ \ $} We will only prove the theorem for the function
$ \ \exp_e \ $. For the other functions the computations are similar.

$ \ \  \exp_e ^{12} \circ   \exp_e ^{13} \circ  \exp_e ^{23} \ (x,y,z) = 
\ \exp_e ^{12} \circ   \exp_e ^{13} (x,ey,e * z) =\\
\ \exp_e ^{12}  (ex,y,e*(e * z)) = (e(ex), y,e*(e * z)) $;

$ \ \exp_e ^{23} \circ  \exp_e ^{13} \circ \exp_e ^{12} (x,y,z) =
\exp_e ^{23} \circ  \exp_e ^{13} (ex,e*y,z) =
\exp_e ^{23}  (e(ex),e*y,e * z) =\\
(e(ex), y,e*(e * z)) $.

\qed

\bigskip

The next theorem presents a Yang-Baxter system. From it,
one could associate a new solution for the 
Quantum Yang-Baxter Equation
(see \cite{ybs}).

\bigskip

\begin{theorem}
 Let $ W(x,y) = (x, \ xy) $ be the usual solution to 
the Quantum Yang-Baxter equation arising from a rack.
Then, 
$ \ \ \ \exp_e ^{23} \circ  \exp_e ^{13} \circ W^{12}  =
W^{12} \circ   \exp_e ^{13} \circ  \exp_e ^{23} $.

Since
$ \ \  \exp_e ^{12} \circ   \exp_e ^{13} \circ  \exp_e ^{23} =
\exp_e ^{23} \circ  \exp_e ^{13} \circ \exp_e ^{12} $, we
obtain a Yang-Baxter system.

\end{theorem}

{\bf Proof. $ \  $} 
$   \exp_e ^{23} \circ  \exp_e ^{13} \circ W^{12}  (x,y,z) =
 \exp_e ^{23} \circ  \exp_e ^{13} (x,xy,z) =
\ \exp_e ^{23} (ex,xy,e*z) =  (ex, e(xy), e*(e*z)) $. 

$ \ \   W^{12} \circ   \exp_e ^{13} \circ  \exp_e ^{23}  (x,y,z) =
 W^{12} \circ   \exp_e ^{13} (x,ey,e*z) =
W^{12} (ex, ey, e*(e*z)) = (ex, (ex)(ey), e*(e*z)) $. 
But, $ \ e(xy) = (ex)(ey) \ $.
\qed
\bigskip
\begin{theorem}
 Let $ W(x,y) = (x, \ xy) $ be the usual solution to 
the Quantum Yang-Baxter equation arising from a rack.
 Let $ Z(x,y) = (x \diamond y, \ y) $ be another solution to 
the Quantum Yang-Baxter equation.
Then, $ \ \ W \ , \ exp_e \ , \ Z $ is a Yang-Baxter sysytem.

\end{theorem}

\end{document}